\documentclass[12pt]{amsart}

\usepackage{amssymb,latexsym}
\usepackage{verbatim,euscript}
\usepackage{fullpage,color}

\usepackage[colorlinks=true,linkcolor=blue,urlcolor=my_color,citecolor=magenta]{hyperref}

\definecolor{my_color}{rgb}{0,0.5,0.5}
\definecolor{MIXT}{rgb}{0.4,0.3,0.6}

\input {cyracc.def}
\numberwithin{equation}{section}

\font\tencyr=wncyr10 %scaled \magstephalf
\def\rus{\tencyr\cyracc}

\newtheorem{thm}{Theorem}[section]
\newtheorem{lm}[thm]{Lemma}%[chapter]
\newtheorem{cl}[thm]{Corollary}%[chapter]
\newtheorem{prop}[thm]{Proposition}%[chapter]

\theoremstyle{remark}
\newtheorem{rmk}[thm]{Remark}

\theoremstyle{definition}
\newtheorem{ex}[thm]{Example}

\newtheorem*{rema}{Remark}

\newenvironment{proof*}
{\noindent {\sl Proof.}\quad }{\hfill $\square$}

\newcommand {\ah}{{\mathfrak a}}
\newcommand {\be}{{\mathfrak b}}
\newcommand {\ce}{{\mathfrak c}}

\newcommand {\g}{{\mathfrak g}}

\newcommand {\gH}{{\eus H}}

\newcommand {\el}{{\mathfrak l}}

\newcommand {\n}{{\mathfrak n}}

\newcommand {\p}{{\mathfrak p}}

\newcommand {\te}{{\mathfrak t}}
\newcommand {\ut}{{\mathfrak u}}

\newcommand {\slno}{{\mathfrak{sl}}_{n+1}}

\newcommand {\spn}{{\mathfrak{sp}}_{2n}}

\newcommand {\esi}{\varepsilon}
\newcommand {\ap}{\alpha}

\newcommand {\lb}{\lambda}

\newcommand {\tap}{{\tilde{\alpha}}}

\newcommand {\HW}{\widehat W}
\newcommand {\HV}{\widehat V}
\newcommand {\HP}{\widehat\Pi}
\newcommand {\HD}{\widehat\Delta}

\newcommand {\cF}{{\mathcal F}}

\newcommand {\BR}{{\mathbb R}}

\newcommand {\BZ}{{\mathbb Z}}

\newcommand {\hot}{{\mathsf{ht}}}

\newcommand {\Ima}{{\mathsf{Im}}}

\newcommand {\rk}{{\mathsf{rk\,}}}

\newcommand {\tri}{\mathfrak{sl}_2}
\newcommand {\GR}[2]{{\textrm{{\bf #1}}}_{#2}}

\newcommand {\Ab}{\mathfrak{Ab}}
\newcommand {\Abo}{\mathfrak{Ab}^o}%{\overset{o}{\mathfrak{Ab}}}

\newcommand {\beq}{\begin{equation}}
\newcommand {\eeq}{\end{equation}}

\newcommand{\curge}{\succcurlyeq}
\newcommand{\curle}{\preccurlyeq}
\renewcommand{\le}{\leqslant}
\renewcommand{\ge}{\geqslant}

\newenvironment{E6}[6]{%
{\small\begin{tabular}{@{}c@{}}
{#1}--{#2}--\lower3.5ex\vbox{\hbox{{#3}\rule{0ex}{2.5ex}}
\hbox{\hspace{0.4ex}\rule{.1ex}{1ex}\rule{0ex}{1.4ex}}\hbox{{#6}\strut}}--{#4}--{#5}
\end{tabular}}}

\newcommand{\eus}{\EuScript}

\begin{document}
\setlength{\parskip}{2pt plus 4pt minus 0pt}
\hfill {\scriptsize December 24, 2015} 
\vskip1.5ex

\title[Normalisers of  abelian ideals]{Normalisers of abelian ideals of a Borel subalgebra and $\BZ$-gradings of a simple Lie algebra}
\author{Dmitri I. Panyushev}
\address[]{Institute for Information Transmission Problems of the R.A.S., Bolshoi Karetnyi per. 19,  
127051 Moscow,  Russia}
\email{panyushev@iitp.ru}
\begin{abstract}

Let $\g$ be a simple Lie algebra  
and  $\Ab$ the poset of all abelian ideals of a fixed Borel subalgebra of $\g$. If $\ah\in\Ab$, then the normaliser of $\ah$ is a standard parabolic subalgebra of $\g$. 
We give an explicit description of the normaliser for a class of abelian ideals that includes all maximal abelian ideals. We also elaborate on a relationship between abelian ideals and $\BZ$-gradings of $\g$ associated with their normalisers.
\end{abstract}
\maketitle

%\tableofcontents
\section*{Introduction}

\noindent
Let $\g$ be a complex simple Lie algebra with a  triangular decomposition 
$\g=\ut\oplus\te\oplus \ut^-$, where $\te$ is a fixed Cartan subalgebra and $\be=\ut\oplus\te$
is a fixed Borel subalgebra.  A subspace $\ah\subset\be$ is an {\it abelian ideal\/} if  $[\be,\ah]\subset \ah$ and 
$[\ah,\ah]=0$. Then $\ah\subset\ut$.
The general theory of abelian ideals of $\be$ is based on their relations with the so-called {\it minuscule 
elements\/} of the affine Weyl group $\HW$, which is due to D.~Peterson 
(see Kostant's account in~\cite{ko98}).  The subsequent development has lead to a number of spectacular results of 
combinatorial and representation-theoretic nature, see e.g.~\cite{cp1,cp2,cp3,ko04,imrn,jems,som05,suter}.

The normaliser of $\ah$ in $\g$, denoted $\n_\g(\ah)$, contains $\be$, i.e., it is a {\it standard\/} 
parabolic subalgebra of $\g$. In this note, we study the normalisers of abelian ideals using the 
corresponding minuscule elements of $\HW$ and $\BZ$-gradings of $\g$.
 
Let $\Delta$ be the root system of $(\g,\te)$, $\Delta^+$ the set of positive roots corresponding to $\ut$,  
$\Pi$ the set of simple roots in $\Delta^+$, and $\theta$  the highest root in  $\Delta^+$. Then $W$ is the Weyl group and $\g_\gamma$ is the root space for $\gamma\in\Delta$.
We write $\Ab=\Ab(\g)$ for the set of all abelian ideals of $\be$ and 
think of $\Ab$ as poset with respect to inclusion. 
Since $\ah\in\Ab$ is a sum of certain root spaces of $\ut$, we often identify such an $\ah$ 
with the corresponding subset $I=I_\ah$ of $\Delta^+$.

Let $\Abo$ denote the set of nonzero abelian ideals and $\Delta^+_l$  the set
of long positive roots.  %In the simply-laced case, all roots are assumed to be long.
In~\cite[Sect.\,2]{imrn}, we defined a surjective mapping 
$\tau: \Abo \to \Delta^+_l$ and studied its fibres. 
If  $\tau(\ah)=\mu$, then $\mu\in\Delta^+_l$ is called the 
{\it rootlet\/} of $\ah$. %, also denoted by $\rt(\ah)$ or $\rt(I_\ah)$.  
Letting $\Ab_\mu=\tau^{-1}(\mu)$, we get a  partition of $\Abo$ parameterised by
$\Delta^+_l$. Each fibre $\Ab_\mu$ is  a sub-poset of $\Ab$.  By~\cite[Sect.\,3]{imrn}, the poset
$\Ab_\mu$ has a unique minimal and unique maximal element for any $\mu\in \Delta^+_l$.  These are denoted by $\ah(\mu)_{min}$ and $\ah(\mu)_{max}$, respectively. The corresponding sets of positive roots are 
$I(\mu)_{\min}$ and $I(\mu)_{\max}$. The abelian ideals of the form $\ah(\mu)_{min}$ (resp.
$\ah(\mu)_{max}$) will be referred to as the {\it root-minimal} (resp. {\it root-maximal}).
The set of globally maximal abelian ideals coincides with $\{\ah(\ap)_{max}\mid \ap\in\Pi_l\}$, where
$\Pi_l=\Delta^+_l\cap \Pi$~\cite[Cor.\,3.8]{imrn}. 

If $\p\supset\be$, then a Levi subalgebra $\el$ of $\p$ is said to be {\it standard}, if $\el\supset\te$.
Set $\p[\mu]_{min}=\n_\g(\ah(\mu)_{min})$ and  $\p[\mu]_{max}=\n_\g(\ah(\mu)_{max})$.
%These are standard parabolic subalgebras. 
Write $\Pi[\mu]_{min}$ for the simple roots of the standard Levi subalgebra of $\p[\mu]_{min}$, and likewise for `max'.
Our main results  are the following:

{\bf I.} \ We explicitly describe $\Pi[\mu]_{min}$ for any root-minimal ideal $\ah(\mu)_{min}$. The answer is 
given in terms of the element $w_\mu\in W$ that takes $\theta$ to $\mu$ and has minimal possible length, 
see Theorem~\ref{thm:norm-mu-min}.  The elements $w_\mu$ have already been considered in~\cite{imrn}, 
and  we 
also provide here new properties of them. Furthermore, if $\theta$ is fundamental and $\ap_\theta\in\Pi$ is 
such that $(\theta,\ap_\theta)\ne 0$, then $\ap_\theta$ is long and we prove that 
$\Pi\setminus\Pi[\ap_\theta]_{min}$ consists of the simple roots that are adjacent to $\ap_\theta$ in the 
Dynkin diagram (Proposition~\ref{thm:norm-min}).

{\bf II.} \  %In Section~\ref{sect:rela}, 
We give a new characterisation of normalisers of arbitrary $\be$-stable subspaces 
of $\ut$ (Theorem\,\ref{thm:new-normalise}) and then explicitly describe the normalisers of the globally maximal abelian ideals, i.e., we determine
$\Pi[\ap]_{max}$ for all $\ap\in \Pi_l$ (Theorem~\ref{thm:norm-max}). This is based on a relationship 
between $\ah(\ap)_{min}$ and $\ah(\ap)_{max}$ for $\ap\in\Pi_l$~\cite[Theorem\,4.7]{jems}, which allows us to retrieve information 
on $\Pi[\ap]_{max}$ from that on $\Pi[\ap]_{min}$.

{\bf III.} \  In Section~4, we relate $\ah\in\Ab(\g)$ to the $\BZ$-grading of $\g$ corresponding to $\n_\g(\ah)$.
%any standard parabolic subalgebra determines a $\BZ$-grading of $\g$, and our aim is to relate $\ah\in\Ab$ . 
Let $\mathfrak{Par}(\g)$ denote the set of 
all standard parabolic subalgebras of $\g$. By Peterson's theorem~\cite{ko98}, $\#\Ab(\g)=2^{\rk\g}$, hence the sets 
$\Ab(\g)$ and $\mathfrak{Par}(\g)$ are equipotent. There is the natural mapping $f_1:\Ab(\g)\to\mathfrak{Par}(\g)$ that takes $\ah$ to $\n_\g(\ah)$. By~\cite{pr}, $f_1$ is a bijection if and only if $\g=\slno$ or $\spn$.
Using the $\BZ$-grading associated with $\p\in\mathfrak{Par}(\g)$, we define here the natural mapping $f_2: \mathfrak{Par}(\g)\to \Ab(\g)$ and prove that $f_2$ is a bijection if and only if $\g=\slno$ or $\spn$; furthermore, $f_2=f_1^{-1}$ for these two series (Theorem~\ref{thm:svoistva-f2}).
We say that $\ah\in\Ab$ is {\it reflexive\/}, if $(f_2\circ f_1)(\ah)=\ah$. Then all abelian ideals for $\slno$ and $\spn$ are reflexive. We also prove that  $\ah(\ap)_{min}$ and $\ah(\ap)_{max}$ ($\ap\in\Pi_l$) are always reflexive and
characterise them in terms of the corresponding $\BZ$-gradings (see Theorem~\ref{thm:udivit1} and 
Remark~\ref{rmk:might-be}).
Finally, we conjecture that the sets $\Ima(f_1\circ f_2)$ and
$\Ima(f_2\circ f_1)$ are always equipotent and the maps $f_1$ and $f_2$ induce the mutually inverse bijections between them.

%Our main object if study is the abelian ideals
%$\ah(\ap)_{min}, \ah(\ap)_{max}$ for $\ap\in\Pi_l$ and their normalisers.

%, then the last property means that
%$I(\mu)_{\min}=I(\mu)_{\max}$ if and only if $(\mu,\theta)=0$. 

%\begin{itemize}
%\item \ $\# I(\mu)_{\min}=(\rho, \theta^\vee-\mu^\vee)+1$, where $\rho=\frac{1}{2}
%      \sum_{\gamma\in\Delta^+} \gamma$ and $\mu^\vee=2\mu/(\mu,\mu)$;
%\item \   $I=I(\mu)_{\min}$  for some $\mu\in\Delta^+_l$ if and only if $I\subset
%      \eus H:=\{\gamma\in \Delta^+ \mid (\gamma, \theta)\ne 0\}$;
%\item \ $I(\mu)_{\min} \subset I(\mu')_{\min}$ if and only if $\mu'\curle \mu$, where `$\curle$' is the usual {\it root order\/} on $\Delta^+$.
%\item \ $I(\mu)_{\min}=I(\mu)_{\max}$ if and only if $(\mu,\theta)=0$ \cite[Thm.\,5.1]{imrn}.
%\end{itemize}

%\noindent
We refer to \cite{bour,hump} for standard results on root systems and (affine) Weyl groups.

{\small
{\bf Acknowledgements.} The research was carried out at the IITP RAS at the expense of the Russian
Foundation for Sciences (project {\rus N0} 14-50-00150).
}

%%%%%%%%%%%%%%%%   
\section{Preliminaries on minuscule elements and normalisers of abelian ideals}
\label{sect:odin}

\noindent
We equip  $\Delta^+$  with the usual partial ordering `$\curle$'.
This means that $\mu\curle\nu$ if $\nu-\mu$ is a non-negative integral linear combination
of simple roots. If $M$ is a subset of $\Delta^+$, then $\min(M)$ and $\max(M)$ are the minimal and maximal
elements of $M$ with respect to ``$\curle$''.
%Write $\mu\prec\nu$ if $\mu\curle\nu$ and $\mu\ne\nu$.

Any $\be$-stable subspace $\ce\subset \ut$ is a sum of certain root spaces in $\ut$,  i.e.,
$\ce=\bigoplus_{\gamma\in I_\ce}\g_\gamma$.  The relation $[\be,\ce]\subset \ce$ is equivalent to
that $I=I_\ce$ is an {\it upper ideal\/} of the poset $(\Delta^+, \curle)$, i.e., 
if $\nu\in I$, $\gamma\in\Delta^+$, and $\nu\curle \gamma$, then $\gamma\in I$.
We mostly work in the combinatorial setting, so that a $\be$-ideal $\ce\subset\ut$
is being identified with the corresponding upper ideal $I$ of $\Delta^+$. 
The property of being abelian additionally means that
$\gamma'+\gamma''\not\in \Delta^+$ for all $\gamma',\gamma''\in I$.

We recall below the notion of a minuscule element of  $\HW$ and their relation to abelian ideals.
We have $\Pi=\{\ap_1,\dots,\ap_n\}$, the vector space $\te_\BR=V=\oplus_{i=1}^n{\mathbb R}\ap_i$, 
the  Weyl group $W$ generated by  simple reflections
$s_1,\dots,s_n$,  and a $W$-invariant inner product $(\ ,\ )$ on $V$. 
Letting $\widehat V=V\oplus {\mathbb R}\delta\oplus {\mathbb R}\lb$, we extend
the inner product $(\ ,\ )$ on $\widehat V$ so that $(\delta,V)=(\lb,V)=
(\delta,\delta)= (\lb,\lb)=0$ and $(\delta,\lb)=1$. Set  $\ap_0=\delta-\theta$, where
$\theta$ is the highest root in $\Delta^+$. 
Then

%$Q^+=\{\sum_{i=1}^n m_i\ap_i \mid m_i=0,1,2,\dots \}$ is the monoid generated by the 
%positive roots.

\begin{itemize}
\item[] \ 
$\widehat\Delta=\{\Delta+k\delta \mid k\in {\mathbb Z}\}$ is the set of affine
(real) roots; 
\item[] \ $\HD^+= \Delta^+ \cup \{ \Delta +k\delta \mid k\ge 1\}$ is
the set of positive affine roots; 
\item[] \ $\HP=\Pi\cup\{\ap_0\}$ is the corresponding set
of affine simple roots;
\item[] \  $\mu^\vee=2\mu/(\mu,\mu)$ is the coroot corresponding to 
$\mu\in \widehat\Delta$.
%\item[] \   $Q=\oplus _{i=1}^n {\mathbb Z}\ap_i$  is the {\it root lattice\/}  and 
%$Q^\vee=\oplus _{i=1}^n {\mathbb Z}\ap_i^\vee$  is the {\it coroot lattice\/} in $V$.
\end{itemize}
 
\noindent
For each $\ap_i\in \HP$, let $s_i=s_{\ap_i}$ denote the corresponding reflection in $GL(\HV)$.
That is, $s_i(x)=x- (x,\ap_i)\ap_i^\vee$ for any $x\in \HV$.
The affine Weyl group, $\HW$, is the subgroup of $GL(\HV)$
generated by the reflections $s_0,s_1,\dots,s_n$.
The extended inner product $(\ ,\ )$ on $\widehat V$ is $\widehat W$-invariant. 
The {\it inversion set\/} of $w\in\HW$ is $\eus N(w)=\{\nu\in\HD^+\mid w(\nu)\in -\HD^+\}$. Note that if
$w\in W\subset \HW$, then $\eus N(w)\subset \Delta^+$.

Following Peterson, we say that $w\in \HW$ is  {\it minuscule\/}, if 
$\eus N(w)=\{-\gamma+\delta\mid \gamma\in I_w\}$ for some  $I_w\subset \Delta$.
One then proves that {\sf (i)} $I_w\subset \Delta^+$, {\sf (ii)} $I_w$ is (the set of roots of) an abelian ideal, and
{\sf (iii)} the assignment 
$w\mapsto I_w$ yields a bijection between the minuscule elements of
$\HW$ and the abelian ideals, see \cite{ko98},  \cite[Prop.\,2.8]{cp1}. 
Conversely, if $\ah\in\Ab$ and $I=I_\ah$, then $w_\ah\in\HW$ stands for the corresponding minuscule element.
Clearly, $\dim\ah=\# I_\ah=\#\eus N(w_\ah)$. %=\ell(w_I)$.
%, where $\ell$ is the usual length function on $\HW$.

Given $\ah\in\Abo$ and  $w_\ah\in\HW$, 
the {\it rootlet\/} of $\ah$ is defined by 
\[
   \tau(\ah)=%\tau(I)=
   w_\ah(\ap_0)+\delta=w_\ah(2\delta-\theta) .
\]
By \cite[Prop.\,2.5]{imrn}, we have $\tau(\ah)\in \Delta^+_l$ and every $\mu\in\Delta^+_l$ occurs in this way.

Let $\el$ be the standard Levi subalgebra of $\p=\n_\g(\ah)$ and $\Pi(\el)\subset\Pi$ the set of simple roots of $\el$. By~\cite[Theorem\,2.8]{norm}, the set  $\Pi(\el)$ is determined by $w_\ah$ as follows:
\[
    \ap\in \Pi(\el) \ \Longleftrightarrow \ w_{\ah}(\ap)\in\HP . 
\]
(Actually, this result of \cite{norm} has been proved for any $\be$-stable subspace $\ce\subset\ut$ in place of 
$\ah$. To this end, one also needs a more general theory of elements of $\HW$ associated with arbitrary
$\be$-stable subspaces of $\ut$~\cite{cp1}.)

An advantage of our situation is that,
for the root-minimal abelian ideals $\ah=\ah(\mu)_{min}$, there is a simple formula for $w_\ah$, which allows us to describe the corresponding normaliser in terms of $\mu$.
We also need the following facts: 

\textbullet\quad  $\#\tau^{-1}(\mu)=1$ (i.e., $\ah(\mu)_{min}=\ah(\mu)_{max}$) if and only if
$(\theta,\mu)\ne 0$~\cite[Theorem\,5.1]{imrn}.

\textbullet\quad   $\ah$ is root-minimal %=\ah(\mu)_{\min}$  for $\mu\in\Delta^+_l$ 
if and only if $I_\ah\subset
     \eus H:=\{\gamma\in \Delta^+ \mid (\gamma, \theta)\ne 0\}$~\cite[Theorem\,4.3]{imrn};

In what follows, it will be important to distinguish the cases whether $\theta$ is fundamental or not, and
whether $(\theta,\mu)=0$ or not. Recall that $\theta$ is fundamental if and only if $\Delta$ is not of type
$\GR{A}{n}$ or $\GR{C}{n}$. One also has $\#(\Pi\cap\gH)=\begin{cases} 2 & \text{ for } \GR{A}{n} \\
1 & \text{ for all other types } \end{cases}$ .  For the classical series, we use the standard notation and numbering for $\Pi$, which seems to be the same in all sources. For instance, for $\GR{A}{n}$, we have
$\ap_i=\esi_i-\esi_{i+1}$ ($i=1,\dots,n$), whence $\Pi\cap\gH=\{\ap_1,\ap_n\}$. For $\GR{E}{6}$, our numbering is \quad 
$\text{\begin{E6}{1}{2}{3}{4}{5}{6}\end{E6}}$; \ hence $\Pi\cap\gH=\{\ap_6\}$.

%\noindent
For  $\gamma\in\Delta$ and $\ap\in\Pi$, $[\gamma:\ap]$ stands for the coefficient of $\ap$ in the expression 
of $\gamma$ via $\Pi$.

%%%%%%%%%%%%%%%%%%%  section 
\section{Normalisers of the root-minimal abelian ideals} 
\label{sect:norm-nu-min}

\noindent
In this section, we describe normalisers of the root-minimal abelian ideals for all  
$\mu\in\Delta^+_l$.

There is a unique element of minimal length in $W$ taking $\theta$ to $\mu$~\cite[Theorem\,4.1]{imrn}, which is denoted by $w_{\mu}$. The ideal $\ah(\mu)_{min}$ is completely determined by $w_{\mu}$.
Namely, $w_{\mu}s_0\in \HW$ is the minuscule element corresponding to $\ah(\mu)_{min}$~\cite[Theorem\,4.2]{imrn}.
We begin with two useful properties of the elements $w_{\mu}$.

\begin{lm}   \label{lm:vspomogat1}
If $\beta\in\Pi$ and $(\beta,\mu)=0$, then $w_{\mu}^{-1}(\beta)\in\Pi$ and 
$(w_{\mu}^{-1}(\beta),\theta)=0$.
\end{lm}
\begin{proof}
It is known that $\eus N(w_{\mu}^{-1})=\{\gamma\in\Delta^+\mid (\gamma,\mu^\vee)=-1\}$~\cite[Theorem\,4.1(2)]{imrn}. 
Therefore $w_{\mu}^{-1}(\beta)\in\Delta^+$.  Assume that $w_{\mu}^{-1}(\beta)=\gamma_1+\gamma_2$ is a sum of positive roots. Then $\beta=w_\mu(\gamma_1)+w_\mu(\gamma_2)$. Without loss of generality, one may assume that $-\nu_1:=w_\mu(\gamma_1)$ is negative. Then 
$\nu_1\in \eus N(w_\mu^{-1})$, hence $(-\nu_1,\mu^\vee)=1$. Consequently,
$(\gamma_1,\theta^\vee)=1$. On the other hand, $0=(\mu,\beta)=(\theta,\gamma_1+\gamma_2)$ and therefore $(\theta, \gamma_2)<0$, which is impossible. Thus, $w_\mu^{-1}(\beta)$ must be simple and $(w_\mu^{-1}(\beta),\theta)=(\beta,\mu)=0$.
\end{proof}
\begin{lm} %[{\cite[Lemma\,4.3]{mics}}]   
\label{lm:vspomogat2}
Suppose that $\theta$ is fundamental and $\ap_\theta\in\Pi$ is not orthogonal to $\theta$. If 
$(\theta,\mu)>0$ and $\theta\ne\mu$, then $w_{\mu}^{-1}(\theta)=\theta-\ap_\theta$; or, equivalently,
$w_\mu(\ap_\theta)=\mu-\theta$.
\end{lm}
\begin{proof}  It is well known and easily verified that $\ap_\theta$ is long and $[\theta:\ap_\theta]=2$ (cf. 
also Theorem~\ref{thm:tap-explicit}(ii)). If $\mu\in\gH\setminus\{\theta\}$, then $[\mu:\ap_\theta]=1$.
By \cite[Section\,1]{mics}, multiplicities of the simple reflections in any reduced 
expression of $w_\mu$ are the same, and they are determined by the coefficients of $\theta-\mu$. 
In particular,
$s_{\ap_\theta}$ occurs only once, since $[\theta-\mu:\ap_\theta]=1$ and $\ap_\theta$ is long.
Moreover, the reduced expressions of $w_\mu$ are in a bijections with the ``root paths'' connecting
$\theta$ with $\mu$ inside $\Delta^+_l$. Since $\theta$ is fundamental, the passage $\theta \leadsto
s_{\theta_\ap}(\theta)$ is the only step down from $\theta$ inside $\Delta^+_l$. Hence any root path leading to
$\mu$ starts with this step. Therefore, every
reduced expression of $w_\mu$ begins with $s_{\ap_\theta}$, and one can write
$w_\mu=w's_{\ap_\theta}$, where $w'$ does not contain factors $s_{\ap_\theta}$. Therefore,
$w_{\mu}^{-1}(\theta)=s_{\ap_\theta}w'^{-1}(\theta)=s_{\ap_\theta}(\theta)=\theta-\ap_\theta$.
\end{proof}

\begin{rema} This is a generalisation of \cite[Lemma\,4.3]{mics}, where the similar assertion is proved for $\mu=\ap_\theta$.
\end{rema}
Recall that $\Pi[\mu]_{min}\subset \Pi$ is the set of simple roots for the standard Levi subalgebra of
$\p[\mu]_{min}$. Since $\theta$ is not fundamental if and only if $\Delta=\GR{A}{n}$ or $\GR{C}{n}$, 
%and all roots of $\GR{C}{n}$ that are not orthogonal to $\theta$ are short, 
the following result covers all the possibilities for $\mu$.  

\begin{thm}   \label{thm:norm-mu-min}
For any $\mu\in \Delta^+_l$, the set\/ $\Pi[\mu]_{min}$ has the following description.
\begin{itemize}
\item[\sf (i)] \ %There is always the inclusion\/
$\Pi[\mu]_{min}\cap \theta^\perp= \{ w_\mu^{-1}(\beta) \mid \beta\in \Pi \ \& \ (\beta,\mu)=0\}
=\{\ap\in\Pi\mid w_\mu(\ap)\in\Pi \ \& \ (\ap,\theta)=0\}$.
\item[\sf (ii)] \ If $(\mu,\theta)= 0$, then
$\Pi[\mu]_{min}=\{ w_\mu^{-1}(\beta) \mid \beta\in \Pi \ \& \ (\beta,\mu)=0\}$. In particular,
$\Pi[\mu]_{min}\subset \theta^\perp$.
%$\beta\in\Pi$ and $(\beta,\mu)=0$, then $\ap:=\tilde w^{-1}(\beta)\in \Pi[\mu]_{min}$. 
%Here $(\ap,\theta)=0$;
\item[\sf (iii)] \  Suppose that  $(\mu,\theta)\ne 0$ (i.e., $\mu\in\gH$) and $\mu\ne\theta$.
\begin{itemize}
\item[a)] \ if $\theta$ is fundamental, then 
$\Pi[\mu]_{min}=\{\ap_\theta\}\cup \{ w_\mu^{-1}(\beta) \mid \beta\in \Pi \ \& \ (\beta,\mu)=0\}$, where
$\ap_\theta$ is the only simple root such that $(\theta,\ap_\theta)\ne 0$;
\item[b)] \ if $\Delta=\GR{C}{n}$, then there is no such long roots $\mu$;
\item[c)] \ if $\Delta=\GR{A}{n}$ and $\mu=\ap_1+\dots +\ap_i=\gamma_i$ ($i<n$) or 
$\ap_j+\dots +\ap_n=\tilde\gamma_j$ ($j>1$), then \\ \indent
$\Pi[{\gamma_i}]_{min}=\{\ap_n\}\cup  \{ w_{\gamma_i}^{-1}(\beta) \mid \beta\in \Pi \ \& \ 
(\beta,\gamma_i)=0\}
=\Pi\setminus \{\ap_1,\ap_i\}$
and 
\\ \indent 
$\Pi[\tilde\gamma_j]_{min}=\{\ap_1\}\cup \{w_{\tilde\gamma_j}^{-1}(\beta) \mid \beta\in \Pi \ \& \ (\beta,\tilde\gamma_j)=0\}
=\Pi\setminus \{\ap_j,\ap_n\}$.
\end{itemize}
\item[\sf (iv)] \  If $\mu=\theta$, then $\Pi[\theta]_{min}=\{\beta\in\Pi\mid (\beta,\theta)=0\}$.
\end{itemize}
\end{thm}
\begin{proof}
Since $w_{\mu} s_0\in\HW$ is the minuscule element corresponding to $I(\mu)_{min}$, the general theory of normalisers of  $\be$-stable subspaces of $\ut$ asserts that
\beq    \label{eq:Levi-norm}
  \ap\in \Pi[\mu]_{min} \ \Longleftrightarrow \ w_{\mu}s_0(\ap)\in\HP ,
\eeq
see \cite[Theorem\,2.8]{norm}. Here one has to distinguish two possibilities:

{\it\bfseries (1)}  \  $w_{\mu}s_0(\ap)\in \Pi$;

{\it\bfseries (2)} \  $w_{\mu}s_0(\ap)=\ap_0=\delta-\theta$.
\\[.7ex]
\textbullet \quad Suppose that $w_{\mu}s_0(\ap)=\beta\in\Pi$. Then $w_{\mu}^{-1}(\beta)=s_0(\ap)\in \Delta$. Hence $s_0(\ap)=\ap$ and therefore $(\theta,\ap)=0$ and
$(\beta,\mu)=(w_{\mu}(\ap),w_{\mu}(\theta))=0$. Thus, 
if  $\ap\in \Pi[\mu]_{min}$ satisfies~{\it\bfseries (1)}, then $w_{\mu}(\ap)=\beta\in\Pi$ and
$(\beta,\mu)=(\theta,\ap)=0$.

Conversely, if $\beta\in\Pi$ and $(\beta,\mu)=0$, then Lemma~\ref{lm:vspomogat1} shows that 
$\ap:=w_{\mu}^{-1}(\beta)\in\Pi$ and $(\ap,\theta)=0$. Hence {\it\bfseries (1)} is satisfied for
$\mu$ and $\ap$.
\\[.7ex]
\textbullet \quad Suppose that $w_{\mu}s_0(\ap)=\ap_0=\delta-\theta$. 
Then $w_{\mu}^{-1}(\delta-\theta)=s_0(\ap)$. Therefore, $\ap\in\Pi_l$ and $s_0(\ap)\ne\ap$, i.e.,
$(\ap,\theta)\ne 0$. More precisely, $\delta-w_{\mu}^{-1}(\theta)=\delta-(\theta-\ap)$, hence 
$w_{\mu}^{-1}(\theta)=\theta-\ap$. The last equality can be rewritten as $\theta=\mu-w_{\mu}(\ap)$.
Therefore, $(\mu,\theta)\ne 0$ and $\mu\ne\theta$. Hence equality~{\it\bfseries (2)} can only occur 
for $\mu\in\gH\setminus \{\theta\}$ and $\ap\in\gH$. Furthermore, 
%This excludes the case of $\GR{C}{n}$, where the only long simple root does not belong to $\gH$. Furthermore, if $\Delta\ne\GR{A}{n}$, then
if $\theta$ is fundamental, then one must have $\ap=\ap_\theta$. 
By Lemma~\ref{lm:vspomogat2}, the equality
$w_{\mu}^{-1}(\theta)=\theta-\ap_\theta$ is then satisfied and we conclude that $\ap_\theta\in \Pi[\mu]_{min}$.
\\ \indent 
This proves parts~(i),(ii),(iiia). 

Parts (iiib) is clear, and (iiic) is obtained by a direct calculation.

(iv) \ Here $\ah(\theta)_{min}=\g_\theta$, and the assertion is obvious.
\end{proof}

\noindent
Theorem~\ref{thm:norm-mu-min} provides a complete description of $\Pi[\mu]_{min}$ for all 
$\mu\in \Delta^+_l$. But for some long simple roots, the assertion can be made even more precise.

\begin{prop}   \label{thm:norm-min}
%For $\tap\in \Pi_l$, the set\/ $\Pi[\tap]_{min}$ has the following description.
If $\theta$ is fundamental and $(\theta,\ap_\theta)\ne 0$, then 
$\Pi[\ap_\theta]_{min}=\{\ap_\theta\}\cup \{ \beta\in \Pi \mid  (\beta,\ap_\theta)=0\}$. 
Therefore, $\Pi\setminus \Pi[\ap_\theta]_{min}$ consists of the simple roots that are adjacent to 
$\ap_\theta$ in the Dynkin diagram.
\end{prop}
\begin{proof}
%Since $\tap=\ap_\theta$ is the only simple root that is not orthogonal to $\theta$,  
By Theorem~\ref{thm:norm-mu-min}(iii), we have $\Pi[\ap_\theta]_{min}=
\{\ap_\theta\}\cup\{ w_{\ap_\theta}^{-1}(\beta) \mid \beta\in \Pi \ \& \ (\beta,\ap_\theta)=0\}$.
Therefore, we are to prove that $w_{\ap_\theta}^{-1}$ permutes the simple roots orthogonal to $\ap_\theta$.
If $\beta\in\Pi$ and $(\beta,\ap_\theta)=0$, then we already know that $w_{\ap_\theta}^{-1}(\beta)\in \Pi$. 
Next, using Lemma~\ref{lm:vspomogat2} with $\mu=\ap_\theta$,  we obtain
\[
    (w_{\ap_\theta}^{-1}(\beta), \ap_\theta)=(\beta, w_{\ap_\theta}(\ap_\theta)) = 
    (\beta, \ap_\theta-\theta)=-(\beta,\theta) .
\]
Since $\beta\ne \ap_\theta$ and $\theta$ is fundamental, this must be zero.
\end{proof}

\noindent
The minuscule elements for the root-maximal abelian ideals do not admit a simple formula. 
Therefore, we cannot explicitly describe $\p[\mu]_{max}$ for all $\mu\in\Delta^+$.  However, if $\mu\in\Pi_l$,  
then  $\ah(\mu)_{min}$ is closely related to $\ah(\mu)_{max}$, and  such a situation is 
considered in the  next section.

%%%%%%%%%%%%%%%%%%%  section 
\section{Normalisers of some root-maximal abelian ideals}
\label{sect:rela} 

\noindent
We begin with a new property of the normaliser of an arbitrary $\be$-stable subspace of $\ut$.
Let $\ce\subset\ut$ be such a subspace and $I_\ce$ the corresponding set of positive roots.  
Being a standard parabolic subalgebra, $\n_\g(\ce)$ is fully determined by the simple roots of the standard Levi subalgebra or, equivalently, by the set
of simple roots $\ap$ such that $\g_{-\ap}\not\in \n_\g(\ce)$.
The following is proved in \cite[Theorem\,3.2]{pr}.

\begin{thm}    \label{thm:old-normalise}
 For any $\be$-stable subspace $\ce\subset\ut$ and $\ap\in\Pi$, we have 
\[
   \g_{-\ap}\not\in \n_\g(\ce)\ \Leftrightarrow \ \exists\, \gamma\in\min(I_\ce) \ \text{ such that }\ \gamma-\ap\in\Delta^+\cup\{0\} .
\]
\end{thm}
\noindent
The point of this result is that it suffices to test only the {\sl minimal\/} roots of $I_\ce$. Note that if 
$\gamma-\ap$ is a root, then $\gamma-\ap\in\Delta^+\setminus I_\ce$.
Our new observation is that it is equally suitable to test only the {\sl maximal\/} roots of $\Delta^+\setminus I_\ce$.
To this end, we first provide an auxiliary assertion.
%which is essentially Lemma~3.1 in \cite{pr}.

\begin{lm}   \label{lm:old-result}
Suppose that $\mu\in\Delta^+$ and $\ap,\tap$ are different simple roots. If $\mu+\ap, \mu+\tap\in\Delta$, then
$\mu+\ap+\tap\in\Delta$.
\end{lm}
\begin{proof}
As $\ap, \tap\in\Pi$, we have $(\tap,\ap)\le 0$. Furthermore, 
$(\mu+\ap)-(\mu+\tap)\not\in\Delta$, hence
\[
   (\mu+\ap,\mu+\tap)=(\mu,\mu)+(\mu,\ap)+(\mu,\tap)+(\ap,\tap)\le 0 .
\]
Since $(\mu,\mu)>0$, the sum contains at least one negative summand.

\textbullet \quad If $(\mu,\ap)<0$, then $(\mu+\tap,\ap)<0$ and we are done.

\textbullet \quad If $(\mu,\tap)<0$, then $(\mu+\ap,\tap)<0$ and we are done.

\textbullet \quad If $(\mu,\ap)=(\mu,\tap)=0$, then $\mu,\ap,\tap$ are short roots. Then
$(\mu+\ap,\mu+\tap)=(\mu,\mu)-(\ap,\tap) \ge (\mu,\mu)-\frac{1}{2}(\mu,\mu)>0$, which shows that this case is impossible.
\end{proof}
\begin{thm}  \label{thm:new-normalise}
Suppose that $\ce\subset \ut$ is $\be$-stable and $\ce\ne \ut$. For $\ap\in\Pi$, we have 
\[
   \g_{-\ap}\not\in \n_\g(\ce)\ \Leftrightarrow \ \exists\, \gamma\in\max(\Delta^+\setminus I_\ce)\ \ \text{ such that }\ \gamma+\ap\in \Delta\ \text{ (and hence }\ \gamma+\ap\in I_\ce\text{)}  . 
\]
\end{thm}
\begin{proof}
 The implication ``$\Leftarrow$'' is obvious.\\
 ``$\Rightarrow'$''. If $\g_{-\ap}\not\in \n_\g(\ce)$, then there is $\mu\in \min(I_\ce)$ such that
 $\mu-\ap\in (\Delta^+\setminus I_\ce)\cup\{0\}$. 
 
 \textbullet \quad If %either $\mu-\ap=0$ or 
 $\mu-\ap\in \max(\Delta^+\setminus I_\ce)$, then $\gamma=\mu-\ap$, and we are done;
 
 \textbullet \quad 
 If $\mu-\ap$ is nonzero and not maximal in $\Delta^+\setminus I_\ce$, then there is an $\tilde\ap\in\Pi$ such that 
 $\mu-\ap+\tilde\ap\in \Delta^+\setminus I_\ce$. Applying Lemma~\ref{lm:old-result} to $\mu-\ap$ shows 
 that $\mu+\tilde\ap$ is a root and then automatically, $\mu+\tilde\ap\in I_\ce$.
 Thus, the pair $\{\mu-\ap,\mu\}$ can be replaced with the ``higher'' pair $\{\mu-\ap+\tilde\ap, \mu+\tilde\ap\}$. Eventually, we obtain a pair whose lower root is maximal in $\Delta^+\setminus I_\ce$.

\textbullet \quad If $\mu=\ap$, then $I_\ce$ contains all positive roots with nonzero coefficient of $\ap$.
Since $\Delta^+\setminus I_\ce\ne \varnothing$, there exists a $\nu\in \Delta^+\setminus I_\ce$ such that $\nu+\ap$ is a root, necessarily in $I_\ce$.
If $\nu\not\in \max(\Delta^+\setminus I_\ce)$, then we can perform the induction procedure of the previous paragraph.
\end{proof}

In the setting of  abelian ideals, there is a special case in which $\max(\Delta^+\setminus I_\ce)$ is 
related to the {\sl minimal\/} roots of another ideal. 

\begin{prop}[{\cite[Theorem\,4.7]{jems}}]   \label{prop:svyaz-JEMS}
For any $\tap\in \Pi_l$, one has 
\[
\gamma\in \min(I(\tap)_{min}) \ \Longleftrightarrow \  \theta-\gamma\in \max(\Delta^+\setminus I(\tap)_{max}).
\]
In particular, if $\rk\Delta>1$ (i.e., $I(\tap)_{min}\ne\{\theta\}$), then  
$\max(\Delta^+\setminus I(\tap)_{max})\subset\gH\setminus\{\theta\}$.
\end{prop}

In the rest of this section, we only consider the abelian ideals with rootlet $\tap\in\Pi_l$. 
Using Theorem~\ref{thm:new-normalise} and Proposition~\ref{prop:svyaz-JEMS}, we are going to compare the normalisers 
$\p[\tap]_{max}=\n_\g(\ah(\tap)_{max})$ and $\p[\tap]_{min}=\n_\g(\ah(\tap)_{min})$. We write $\eus S[\tap]_{max}$ and $\eus S[\tap]_{min}$, respectively, for the simple roots that do {\bf not} belong to their standard Levi subalgebras.
In other words, $\eus S[\tap]_{min}:=\Pi\setminus \Pi[\tap]_{min}$, and likewise for `max'.

\begin{thm}           \label{thm:inclusion-min-max}
Far any $\tap\in\Pi$, we have $\eus S[\tap]_{max}\subset \eus S[\tap]_{min}$ and 
thereby\/ $\p[\tap]_{max}\supset \p[\tap]_{min}$.
\end{thm}
\begin{proof}
If $\g\ne \tri$, then $[\ut,\ut]\ne 0$. Hence $\ah(\tap)_{max}\ne \ut$, i.e., $I(\tap)_{max}\ne \Delta^+$. Therefore,
$\ap\in \eus S[\tap]_{max}$ if and only if there exists $\gamma\in \max(\Delta^+\setminus I(\tap)_{max})$ such 
that $\gamma+\ap\in I(\tap)_{max}$ (Theorem~\ref{thm:new-normalise}).
Then $\gamma\in\gH\setminus\{\theta\}$ (Proposition~\ref{prop:svyaz-JEMS})
and hence $\gamma+\ap\in \gH\cap I(\tap)_{max}=I(\tap)_{min}$~\cite[Proposition\,3.2]{jems}.
By Proposition~\ref{prop:svyaz-JEMS}, we have $\nu:=\theta-\gamma\in \min(I(\tap)_{min})$ and 
$\nu-\ap=\theta-(\gamma+\ap)$ is either a root or zero. In both cases,
%Hence $\nu-\ap \not\in I(\tap)_{min}$.
applying Theorem~\ref{thm:old-normalise} to $\nu$, we conclude that $\ap\in \eus S[\tap]_{min}$.
\end{proof}
 
Actually, there is a more precise statement.

\begin{thm}           \label{thm:inclusion-min-max2}
Excluding the case in which $\Delta$ is of type $\GR{A}{n}$ with $\tap=\ap_1$ or $\ap_n$, we have
$\eus S[\tap]_{max}= \eus S[\tap]_{min}\cap \theta^\perp$.
\end{thm}
\begin{proof}
{\bf 1.} Suppose that $\ap\in \eus S[\tap]_{\max}$ and $\gamma\in \max(\Delta^+\setminus I(\tap)_{max})$ is such that 
$\gamma+\ap\in I(\tap)_{max}$. As explained in the previous proof, we then have  
$\nu=\theta-\gamma\in \min(I(\tap)_{min})\subset\gH$ and $\nu-\ap\in \Delta^+\cup\{0\}$.  Consider these two possibilities for $\nu-\ap$.

{\sf (i)} \ $\nu=\ap$. Then $\ap\in I(\tap)_{min}$, which is only possible if $\tap=\ap$, since $I(\tap)_{min}\subset \{\mu\in\Delta^+\mid   \mu\curge \tap\}$ \cite[Proposition\,3.4]{jems}. Therefore
$\tap=\ap$, $\tap\in\gH$, and $[\theta:\tap]=1$. All this only occurs for $\Delta$ of type 
$\GR{A}{n}$ with $\tap=\ap_1$ or $\ap_n$.

{\sf (ii)} \ $\nu-\ap\in \Delta^+$. Then $\nu-\ap\in\gH$, since $(\nu-\ap)+(\gamma+\ap)=\theta$.
That is both $\nu$ and $\nu-\ap$ belong to $\gH\setminus\{\theta\}$. Hence $(\theta,\ap)=0$.

{\bf 2.}  Conversely, assume that  $\ap\in \eus S[\tap]_{min}\cap \theta^\perp$. That is, $(\theta,\ap)=0$ and
for some $\nu\in \min(I(\tap)_{min})$, we have $\nu-\ap\in \Delta^+\cup\{0\}$.

For $\nu=\ap$, we argue as in part {\bf 1}(i). If $\nu-\ap\in\Delta^+$, then both
$\gamma=\theta-\nu$ and $\gamma+\ap$ are roots, and 
$\gamma\in \max (\Delta^+\setminus I(\tap)_{max})$ in view of Proposition~\ref{prop:svyaz-JEMS}.
Hence $\ap\in \eus S[\tap]_{max}$.
\end{proof}

\begin{rmk}   \label{rem:except-case}
Recall that $\ah(\tap)_{min}=\ah(\tap)_{max}$ if and only if $(\tap, \theta)\ne 0$, i.e., 
$\tap\in\gH$~\cite[Theorem\,5.1(i)]{imrn}. If this is the case (and $\Delta\ne\GR{A}{n}$), then
Theorem~\ref{thm:inclusion-min-max2} implies that $\eus S[\tap]_{max}=\eus S[\tap]_{min}\subset
\theta^\perp$. In the distinguished case of $(\GR{A}{n}, \ap_1 \text{ or }  \ap_n)$, we have
$\ah(\ap_1)_{min}=\ah(\ap_1)_{max}$ and $\eus S[\ap_1]_{min}=\{\ap_1\}$, whereas $\Pi\cap\gH=\{\ap_1,\ap_n\}$.
\end{rmk}
 
\begin{cl}  \label{cor:raznye}
If $I(\tap)_{min}\ne I(\tap)_{max}$, then $\p[\tap]_{min}\ne \p[\tap]_{max}$.
\end{cl}
\begin{proof}
Since $I(\tap)_{min}\ne I(\tap)_{max}$, we have $(\tap,\theta)=0$. Then 
$\Pi[\tap]_{min}\subset \theta^\perp$ by Theorem~\ref{thm:norm-mu-min}(ii).  Then $\eus S[\tap]_{min} \supset \Pi\cap \gH$, and
$\eus S[\tap]_{max}\cap \gH=\varnothing$ in view of Theorem~\ref{thm:inclusion-min-max2}. That is,
$\eus S[\tap]_{min}\ne \eus S[\tap]_{max}$.
\end{proof}

\noindent
Combining Theorems~\ref{thm:norm-mu-min} and \ref{thm:inclusion-min-max2} yields a complete
description of the normaliser for the maximal abelian ideals $\ah(\tap)_{max}$, which turns out to be more 
uniform than that for $\ah(\tap)_{min}$. In the rest of the section, we write $\tilde w$ in place of 
$w_{\tap}$.

\begin{thm}    \label{thm:norm-max}
%\begin{itemize}
{\sf (i)}\ 
Excluding the case in which $\Delta$ is of type $\GR{A}{n}$ with $\tap=\ap_1$ or $\ap_n$, we have
\[
   \Pi[\tap]_{max}=(\Pi\cap \gH) \bigsqcup \{ \tilde w^{-1}(\beta) \mid \beta\in \Pi \ \& \ (\beta,\tap)=0\} . 
\]
\indent {\sf (ii)} \ In particular, if $(\theta,\tap)=0$, then 
$\Pi[\tap]_{max}=(\Pi\cap \gH)\sqcup \Pi[\tap]_{min}$;

{\sf (iii)} \ In particular, if $\theta$ is fundamental and $(\theta,\tap)\ne 0$, then  \\ 
\centerline{$\Pi[\tap]_{max}=\Pi[\tap]_{min}=\{\tap\} \sqcup \{\beta\in\Pi\mid (\beta, \tap)=0\}$.}
\end{thm}
%\begin{proof}\end{proof}

\noindent
Let us say that $\beta\in\Pi$ is {\it admissible\/} (for $\tap$) if $(\beta,\tap)=0$. It follows from 
Theorem~\ref{thm:norm-mu-min} that an admissible root always gives rise to a simple root of the Levi 
subalgebra of $\p[\tap]_{min}$. Furthermore, if $\theta$ is fundamental and $(\tap,\theta)\ne 0$, then
$\tap$ also belongs to $\Pi[\tap]_{min}$.

\begin{ex}
{\bf (1)} \ $\Delta=\GR{A}{n}$, $\tap=\ap_2$. Here 
$\tilde w=s_1s_3\dots s_n$ and the admissible roots are $\ap_4,\dots,\ap_n$. One has
$\tilde w^{-1}(\ap_i)=\ap_{i-1}$ for them. Hence $\Pi[\ap_2]_{min}=\{\ap_3,\ap_4,\dots,\ap_{n-1}\}$
and $\eus S[\ap_2]_{min}=\{\ap_1,\ap_2,\ap_n\}$. Then $\eus S[\ap_2]_{max}=\{\ap_2\}$.
% by Theorem~\ref{thm:inclusion-min-max2}.

More generally, for $\tap=\ap_i$ ($2\le i\le n-1$), one obtains
$\eus S[\ap_i]_{min}=\{\ap_1,\ap_i,\ap_n\}$ and $\eus S[\ap_i]_{max}=\{\ap_i\}$.

{\bf (2a)} \ $\Delta=\GR{D}{4}$, $\tap=\ap_1$. Here $\tilde w=s_2s_3s_4s_2$ and the admissible 
roots are $\ap_3,\ap_4$. One has $\tilde w^{-1}(\ap_3)=\ap_{4}$ and $\tilde w^{-1}(\ap_4)=\ap_{3}$.
Hence $\eus S[\ap_1]_{min}=\{\ap_1,\ap_2\}$ and $\eus S[\ap_1]_{max}=\{\ap_1\}$.

{\bf (2b)} \ $\Delta=\GR{D}{4}$, $\tap=\ap_2$. There is no admissible roots here, hence $\tilde w$ is 
not really needed. Since $(\ap_2,\theta)\ne 0$, we have 
$\eus S[\ap_2]_{min}=\eus S[\ap_2]_{max}=\{\ap_1,\ap_3,\ap_4\}=\Pi\setminus (\Pi\cap\gH)$.

{\bf (3)} \ $\Delta=\GR{C}{n}$, $\tap=\ap_n$ (the only long simple root). Here $\tilde w=
s_{n-1}\dots s_2s_1$ and the admissible roots are $\ap_1,\dots,\ap_{n-2}$. One has
$\tilde w^{-1}(\ap_i)=\ap_{i+1}$ for them. Hence $\Pi[\ap_n]_{min}=\{\ap_2,\ap_3,\dots,\ap_{n-1}\}$
and $\eus S[\ap_n]_{min}=\{\ap_1,\ap_n\}$. Then $\eus S[\ap_n]_{max}=\{\ap_n\}$.

{\bf (4a)} \ $\Delta=\GR{E}{6}$, $\tap=\ap_3$. Here $\tilde w=
s_6s_4s_2s_5s_3s_1s_2s_4s_3s_6$ and the admissible roots are $\ap_1,\ap_5$. 
One has $\tilde w^{-1}(\ap_1)=\ap_{4}$ and $\tilde w^{-1}(\ap_5)=\ap_{2}$.
Hence $\eus S[\ap_3]_{min}=\{\ap_1,\ap_3,\ap_5, \ap_6\}$ and $\eus S[\ap_3]_{max}=\{\ap_1,\ap_3,\ap_5\}$.

{\bf (4b)} \ $\Delta=\GR{E}{6}$, $\tap=\ap_2$. Here $\tilde w=
s_3s_6s_4s_5s_3s_1s_2s_4s_3s_6$ and the admissible roots are $\ap_4,\ap_5,\ap_6$. One has 
$\tilde w^{-1}(\ap_4)=\ap_{3}$, $\tilde w^{-1}(\ap_5)=\ap_{2}$ and $\tilde w^{-1}(\ap_6)=\ap_{5}$.
Hence $\eus S[\ap_2]_{min}=\{\ap_1,\ap_4, \ap_6\}$ and $\eus S[\ap_2]_{max}=\{\ap_1,\ap_4\}$.
\end{ex}

%%%%%%%%%%%%%%%%%%%  section 
\section{Normalisers of abelian ideals and $\BZ$-gradings}
\label{sect:udivit} 

\noindent
In this section, we elaborate on a relationship between the abelian ideals, their normalisers and the associated $\BZ$-gradings.
Any subset $S\subset \Pi$ gives rise to a $\BZ$-grading of $\g$.
%For $\ap\in\Pi$, we 
Set $\deg(\ap)=\begin{cases} 0, & \ap\in \Pi\setminus S\\ 1, & \ap\in S \end{cases}$, 
and extend it to the whole of $\Delta$ by linearity. Then the $\BZ$-grading 
$\g=\bigoplus _{i\in\BZ}\g(i)$ is defined by the requirement that $\te\subset \g(0)$ and
$\g_\gamma\subset \g(\deg(\gamma))$ for any $\gamma\in\Delta$. 
Set $\g({\ge} j)=\bigoplus_{i\ge j} \g(i)$.
If we wish to make the dependance on $S$ explicit, then we write $\g(i;S)$ and $\g({\ge}j; S)$.

Let $\p$ be a standard parabolic subalgebra, $\el$ the standard Levi subalgebra of $\p$, 
and $\Pi(\el)$ the set of simple roots of $\el$. Then $S=S(\p)=\Pi\setminus \Pi(\el)$ determines the
$\BZ$-grading {\it associated with\/} $\p$, and we also write $\p=\p(S)$.
In this case,
$\g(0; S)=\el$, $\g({\ge}0; S)=\p$,  and $\g({\ge}1; S)$ is the nilradical of $\p$.

The {\it height\/} of a $\BZ$-grading 
is the maximal $i$ such that $\g(i)\ne \{0\}$. For $S=\Pi\setminus \Pi(\el)$, we also say that it is the
{\it height of\/} $\p(S)$, denoted $\hot(\p(S))$.
%\\ \indent
It is easily seen that 
$\hot(\p(S))=\deg(\theta)=\sum_{\ap\in S} [\theta:\ap]$.
Clearly, if $j\ge [\hot(\p)/2]+1$, then $\g({\ge} j)$ is an abelian ideal of $\be$.

{\bf Convention.} If $(\theta,\tap)\ne 0$, then $I(\tap)_{min}=I(\tap)_{max}$. In this case, we omit
the subscripts `min' and `max' from the notation for all relevant objects; that is, we merely write
$\p[\tap]$, $\eus S[\tap]$, etc. %Recall also that $\tilde w=w_{\tap}$, hence $\tilde w(\theta)=\tap$.

\begin{thm}   \label{thm:tap-explicit}
Suppose that $\theta$ is fundamental, with the corresponding $\ap_\theta\in\Pi$. %and $(\tap,\theta)\ne 0$. Then
\begin{itemize}
%\item[\sf (i)] \ $I(\tap)_{min}=I(\tap)_{max}$;
\item[\sf (i)] \ $\eus S[\ap_\theta]=\{\beta\in \Pi\setminus \{\ap_\theta\} \mid (\beta,\ap_\theta)\ne 0\}$, the set of all 
simple roots adjacent to $\ap_\theta$;
\item[\sf (ii)] \ $\ap_\theta$ is long, $[\theta:\ap_\theta]=2$, and \ $\hot(\p[\ap_\theta])=3$;
\item[\sf (iii)] \  $\ah(\ap_\theta)=\g({\ge}2; \eus S[\ap_\theta])$.
\end{itemize}
\end{thm}
\begin{proof}
 (i)   It is already proved in Proposition~\ref{thm:norm-min}.
 \\
 (ii)  If $\theta$ is fundamental, then $(\theta,\ap_\theta^\vee)=1=(\ap_\theta,\theta^\vee)$. Hence $\ap_\theta$ is necessarily long. Furthermore, 
 \[
    (\theta,\theta)=(\theta,\sum_{\ap\in\Pi} [\theta:\ap]\ap)=[\theta:\ap_\theta](\theta,\ap_\theta)=\frac{1}{2}[\theta:\ap_\theta](\theta,\theta) .
 \]
 Hence $[\theta:\ap_\theta]=2$. Finally,
 \[
   1=(\theta,\ap_\theta^\vee)=2[\theta:\ap_\theta]-\sum_{\beta \ \text{adjacent}}[\theta:\beta] ,
 \] 
where the sum ranges over the simple roots $\beta$ adjacent to $\ap_\theta$ in the Dynkin diagram.
Therefore,  $3=\sum_{\beta \ \text{adjacent}}[\theta:\beta]=\hot(\p[\ap_\theta])$. 

(iii) \  A general description of the minimal roots for all root-minimal ideals $\ah(\mu)_{min}$ is provided 
in~\cite[Prop.\,4.6]{imrn}. In the situation with $\mu=\ap_\theta$, this yields
\[
  \min(I(\ap_\theta))= %\min(I(\tap)_{min})=
  \{ w_{\ap_\theta}^{-1}(\ap_\theta+\beta_i)\mid \beta_i\in\Pi \ \& \ \beta_i \text{ is adjacent to }\ \ap_\theta \} . 
\]
Set $\nu_i=w_{\ap_\theta}^{-1}(\ap_\theta+\beta_i)= \theta+w_{\ap_\theta}^{-1}(\beta_i)$ and write
$\nu_i=m\ap_\theta+ \sum_j m_j\beta_j + \text{(others)}$. Then $m=1$, since
$m=(\nu_i,\theta^\vee)=(\theta+w_{\ap_\theta}^{-1}(\beta_i),\theta^\vee)=2-1=1$. Next, using 
Lemma~\ref{lm:vspomogat2} with $\mu=\ap_\theta$, we obtain
\[
   (\nu_i,\ap_\theta^\vee)=(\theta+w_{\ap_\theta}^{-1}(\beta_i),\ap_\theta^\vee)=1+(\beta_i, \ap_\theta^\vee-\theta^\vee)=1-1=0 .
\]
On the other hand,
\[
   (\nu_i,\ap_\theta^\vee)=2m-\sum_j m_j .
\]   
Therefore, $\sum_j m_j=2$ and all minimal roots belong to $\g(2; \eus S[\ap_\theta])$. Since 
$\g({\ge}2; \eus S[\ap_\theta])$ is an 
abelian ideal and $\ah(\ap_\theta)$ is maximal abelian, we must have $\g({\ge}2; \eus S[\ap_\theta])=\ah(\ap_\theta)$.
\end{proof}

Theorem~\ref{thm:tap-explicit} is a particular case of the following general assertion.

\begin{thm}   \label{thm:udivit1}  \leavevmode\par
\begin{itemize}
\item[\sf (i)] \
For any $\tap\in\Pi_l$ and $n_\tap:=[\theta:\tap]$, we have
$\hot(\p[\tap]_{max})=2n_\tap-1$ and $\ah(\tap)_{max}=\g({\ge}n_\tap; \eus S[\tap]_{max})$.
\item[\sf (ii)] \
If $(\tap,\theta)=0$ (and hence $\eus S[\tap]_{max}\ne \eus S[\tap]_{min}$), then
$\hot(\p[\tap]_{min})=2n_\tap+1$ and $\ah(\tap)_{min}=\g({\ge}n_\tap+1;\eus S[\tap]_{min})$.
\end{itemize} 
\end{thm}
\begin{proof}
Our proof for both parts consists of a case-by-case verification. Using explicit information on 
$\min(I(\tap)_{min})$ and $\min(I(\tap)_{max})$ or results of Section~\ref{sect:rela}, we explicitly 
determine $\eus S[\tap]_{min}$ and $\eus S[\tap]_{max}$. This yields the associated $\BZ$-gradings and 
height of all parabolics involved. 
The minimal roots of $I(\tap)_{min}$ can be determined with the help of \cite[Prop.\,4.6]{imrn},
whereas the minimal roots of $I(\tap)_{max}$ ("generators") are indicated in \cite[Tables~I,II]{pr}.
Then one verifies that the sets $\min(I(\tap)_{min})$ and $\min(I(\tap)_{max})$ always 
coincide with the set of minimal roots of $\g({\ge}n_\tap+1;\eus S[\tap]_{min})$ and
$\g({\ge}n_\tap;\eus S[\tap]_{max})$, respectively.
\end{proof}

\begin{rmk}
We can directly explain the following outcome of Theorem~\ref{thm:udivit1}:
\\[.7ex]
\centerline{\it 
If $(\tap,\theta)=0$, then $\hot(\p[\tap]_{min})=\hot(\p[\tap]_{max})+2$.}

\noindent 
For, by Theorem~\ref{thm:norm-max}(ii), we know that $\eus S[\tap]_{min}=(\Pi\cap\gH)\cup \eus S[\tap]_{max}$. Hence
\[
\hot(\p[\tap]_{min})-\hot(\p[\tap]_{max})=\sum_{\beta\in \Pi\cap\gH} n_\beta .
\]
If $\theta$ is fundamental, then $\Pi\cap\gH=\{\ap_\theta\}$ and $n_{\ap_\theta}=2$ 
(Theorem~\ref{thm:tap-explicit}(ii)). For $\GR{A}{n}$, we have $\Pi\cap\gH=\{\ap_1,\ap_n\}$ and  
$n_{\ap_1}+n_{\ap_n}=2$.
This does not apply to $\GR{C}{n}$, where $(\tap,\theta)\ne 0$ for the unique long simple root $\tap$.
\end{rmk}

\begin{ex}   \label{ex:n-tap=1}
If $n_\tap=1$, then  $I(\tap)_{max}=\{\gamma\in\Delta^+\mid [\gamma:\tap]=1\}$ and 
$\p[\tap]_{max}$ is the maximal parabolic subalgebra with $\eus S[\tap]_{max}=\{\tap\}$. Here 
$\hot(\p[\tap]_{max})=1$. Hence 
Theorem~\ref{thm:udivit1}(i) is satisfied here. Furthermore, if $\theta$ is fundamental and
$(\theta,\ap_\theta)\ne 0$, then $\tap\ne\ap_\theta$ (because $n_{\ap_\theta}=2$), 
$(\theta,\tap)=0$, and $\eus S[\tap]_{min}=\{\tap,\ap_\theta\}$, see Theorem~\ref{thm:norm-max}(ii).
Therefore $\hot(\p[\tap]_{min})=3$, and I can prove {\sl a priori\/} that 
$\ah(\tap)_{min}=\g({\ge}2; \{\tap,\ap_\theta\})$. (As this is not a decisive step, the proof is omitted.)

That is, in principle, there is a better proof of Theorem~\ref{thm:udivit1} if $n_\tap=1$ or  $\tap=\ap_\theta$.
\end{ex}

%\begin{thm}   \lab{thm:udivit2} \end{thm}

Now, we consider arbitrary abelian ideals of $\be$. 
Let $\mathfrak{Par}(\g,\be)=\mathfrak{Par}(\g)$ be the set of all standard parabolic subalgebras of 
$\g$. If $\ah\in\Ab(\g)$, then $\n_\g(\ah)\in \mathfrak{Par}(\g)$. It is proved in \cite{pr} that the assignment 
$\ah\mapsto f_1(\ah)=\n_\g(\ah)$ sets up a bijection 
$\Ab(\g)\stackrel{f_1}{\longrightarrow} \mathfrak{Par}(\g)$ if and only if $\Delta$ is of type $\GR{A}{n}$ or $\GR{C}{n}$ (i.e., $\theta$ is not fundamental). 

Here we extend that observation by looking at a natural mapping in the opposite direction. 
For $\p\in \mathfrak{Par}(\g)$ and the associated $\BZ$-grading, we set
\[
   f_2(\p)=\g({\ge} [\hot(\p)/2]+1) \in \Ab(\g) .
\]
This mapping occurs implicitly in Theorem~\ref{thm:udivit1}, %and \ref{thm:udivit2}, 
where $\hot(\p)$ appears to be always odd.

\begin{thm}  \label{thm:svoistva-f2}   \leavevmode\par
\begin{itemize}
\item[\sf (i)] \ If $\Delta$ is of type $\GR{A}{n}$ or\/ $\GR{C}{n}$, then $f_2:\mathfrak{Par}(\g)\to \Ab(\g)$ is a bijection. Moreover,
$f_2=f_1^{-1}$;
\item[\sf (ii)] \ If $\theta$ is fundamental, then $f_2$ is {\bf not} a bijection. In fact, there is a uniform construction 
of two different $\p_1,\p_2\in \mathfrak{Par}(\g)$ such that $f_2(\p_1)= f_2(\p_2)$.
\end{itemize}
\end{thm}
\begin{proof}
(i) First, we recall the (slightly modified) construction of the bijection $f_1$ for $\GR{A}{n}$.
% and $\GR{C}{n}$. 
For $\ah\in\Ab(\slno)$, let $\min(I_\ah)=\{\gamma_1,\dots,\gamma_k\}$ with
$\gamma_t=\ap_{i_t}+\ap_{i_{t+1}}+\dots +\ap_{j_t}$, where $i_t\le j_t$. Assuming that 
$i_1\le i_2\le \dots\le i_k$, we actually obtain the restrictions 
\[
   1\le i_1< i_2 <\dots < i_k\le j_1<\dots < j_k\le n 
\]
and thereby the bijection between $\Ab(\slno)$ and the subsets of $[n]=\{1,\dots,n\}$. Here one obtains a 
subset of odd (resp. even) cardinality if $i_k = j_1$ (resp. $i_k<j_1$). Moreover, if 
$\p=\n_\g(\ah)$, then it follows from Theorem~\ref{thm:old-normalise} that
$S=S(\p)=\{\ap_{i_1},\ap_{i_2}, \dots,\ap_{i_k},\ap_{j_1}, \dots ,\ap_{j_k}\}$, modulo the possible coincidence of $i_k$ and $j_1$.

Suppose that $\p\in \mathfrak{Par}(\slno)$ and $\# S$ is odd,
$S\sim \{t_1,t_2,\dots, t_{2k-1}\}\subset [n]$, with $t_1< \dots <t_{2k-1}$. Then $\hot(\p)=2k-1$ and the minimal roots of
$\g({\ge} k; S)$ are in a bijection with the shortest intervals of $[n]$ that contain $k$ elements of $S$.
Therefore, these minimal roots are
\begin{gather*}
\gamma_1=\ap_{t_1}+\ap_{t_2}+\dots +\ap_{t_k} , \quad
\gamma_2=\ap_{t_2}+\ap_{t_3}+\dots +\ap_{t_{k+1}} , \\
\quad \dots, \quad
\gamma_k=\ap_{t_k}+\ap_{t_{k+1}}+\dots +\ap_{t_{2k-1}} ,
\end{gather*}
and it is immediate that, for the abelian ideal $\ah=f_2(\p)$ generated by 
$\gamma_1,\dots,\gamma_k$, we have $f_1(\ah)=\p$.

If $\# S$ is even,
$S\sim \{t_1,t_2,\dots, t_{2k}\}\subset [n]$, then $\hot(\p)=2k$ and the minimal roots of
$\g({\ge} k+1; S)$ are
\begin{gather*}
\gamma_1=\ap_{t_1}+\ap_{t_2}+\dots +\ap_{t_{k+1}} , \quad
\gamma_2=\ap_{t_2}+\ap_{t_3}+\dots +\ap_{t_{k+2}} , \\
\quad \dots, \quad
\gamma_k=\ap_{t_k}+\ap_{t_{k+1}}+\dots +\ap_{t_{2k}} .
\end{gather*}
Here again one obtains $\ah=f_2(\p)$ such that $f_1(\ah)=\p$.

We omit the part related to $\GR{C}{n}$, since it goes along the same lines, using the explicit 
description of $f_1$ given in \cite[Theorem\,3.3]{pr}. The point is that the unfolding  
$\GR{C}{n} \leadsto \GR{A}{2n-1}$ (see picture below) yields the identification of $\Ab(\spn)$ and $\mathfrak{Par}(\spn)$ with the
symmetric (with respect to the middle) subsets of $[2n-1]$, and one can use a symmetrised version of the previous argument.

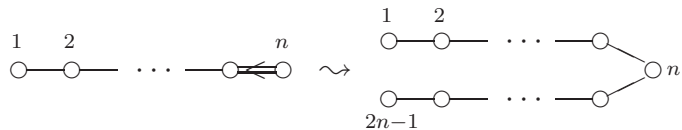
\begin{figure}[htb]
\begin{picture}(300,35)(20,-5)
\multiput(30,12)(20,0){2}{\circle{6}}
\multiput(110,12)(20,0){2}{\circle{6}}
\multiput(112.5,11)(0,2){2}{\line(1,0){15}}
\multiput(33,12)(20,0){2}{\line(1,0){14}}
\put(93,12){\line(1,0){14}}
\put(74,9){$\cdots$}
\put(115,9){$<$} 
\put(27,22){$_1$}
\put(47,22){$_2$}
\put(127,22){$_n$}
\put(144,9){$\leadsto$}

\multiput(170,1)(20,0){2}{\circle{6}}
\multiput(173,1)(20,0){2}{\line(1,0){14}}
\put(214,-2){$\cdots$}
\multiput(170,23)(20,0){2}{\circle{6}}
\multiput(173,23)(20,0){2}{\line(1,0){14}}
\put(214,20){$\cdots$}
\multiput(250,1)(0,22){2}{\circle{6}}
\multiput(233,1)(0,22){2}{\line(1,0){14}}
\put(270,12){\circle{6}}
\put(252.5,2){\line(2,1){14.2}}
\put(252.5,22){\line(2,-1){14.2}}
\put(167,32){$_1$}
\put(187,32){$_2$}
\put(275,12){$_{n}$}
\put(161,-8){$_{2n{-}1}$}
\end{picture}
\end{figure}

(ii) 
Our goal is to produce two different subsets $S_1,S_2\subset \Pi$ such that $\p(S_1)$ and
$\p(S_2)$ give rise to the same abelian ideal. Below we use Theorem~\ref{thm:tap-explicit} and its 
proof. 
\\ \indent
As usual, $\ap_\theta$ is the only simple root that is not orthogonal to $\theta$.  
Let $S_1$ be the set of all simple roots adjacent to $\ap_\theta$ and $S_2=S_1\cup\{\ap_\theta\}$. Then
$\p(S_1)=\p[\ap_\theta]$, $\hot(\p[\ap_\theta])=3$, and $\ah(\ap_\theta)=\g({\ge}2; S_1)$.  Since $n_{\ap_\theta}=2$, we have
$\hot(\p(S_2))=2+\hot(\p[\ap_\theta])=5$ and $\g({\ge}3; S_2)$ is an abelian ideal. The proof of Theorem~\ref{thm:tap-explicit} shows that if $\nu_i\in \min(I(\ap_\theta))$, then $[\nu_i:\ap_\theta]=1$ and 
$\sum_{\beta\in S_1} [\nu:\beta]=2$. Hence $\g_{\nu_i}\in \g(3; S_2)$ and $\ah(\ap_\theta) \subset
\g({\ge}3; S_2)$. As $\ah(\ap_\theta)$ is maximal abelian, one has the equality and therefore
$f_2(\p(S_1))=f_2(\p(S_2))$.
\end{proof}

\begin{rmk}[Some speculations]     \label{rmk:might-be}
Set $\cF=f_1\circ f_2$ and $\tilde\cF=f_2\circ f_1$.
We say that $\ah\in\Ab(\g)$ is {\it reflexive}, if $\tilde\cF(\ah)=\ah$; likewise,  
$\p\in\mathfrak{Par}(\g)$ is {\it reflexive}, if $\cF(\p)=\p$. It is easily seen that $\cF(\p)\supset\p$ for all $\p$, 
while it can happen that $\tilde\cF(\ah)\not\supset\ah$
for some $\ah$ (e.g. if $\g=\GR{E}{6}$). 
\\ \indent
For $\slno$ and $\spn$, all abelian ideals 
are reflexive, whereas this is certainly not the case for the other simple types. However, 
%$\ah\in\Ab(\g)$ such that $(f_2\circ f_1)(\ah)=\ah$. For instance, 
Theorem~\ref{thm:udivit1} implies that the ideals $\ah(\tap)_{min}$ and $\ah(\tap)_{max}$ 
($\tap\in\Pi_l$) are always reflexive. It might be interesting to explicitly determine all reflexive abelian ideals. 
\\ \indent
Our calculations with $\g$ up to rank $4$ suggest that it also might be true that (the restrictions of) $f_1$ 
and $f_2$ induce the mutually inverse bijections between $\Ima(\tilde\cF)\subset \Ab(\g)$ and 
$\Ima(\cF)\subset \mathfrak{Par}(\g)$; in particular, $\#\Ima(\cF)=\#\Ima(\tilde\cF)$. But the equality
$\#\Ima(f_1)=\#\Ima(f_2)$ is false in general (e.g. for $\g=\mathfrak{so}_9$).
\\ \indent
We also conjecture that $\Ima(\cF)=\{\p\mid \cF(\p)=\p\}$ and $\Ima(\tilde\cF)=\{\ah\mid \tilde\cF(\ah)=\ah\}$; 
in other words, $\cF^2=\cF$ and $\tilde\cF^2=\tilde\cF$ in the rings of endomorphisms of the finite sets 
$\mathfrak{Par}(\g)$ and $\Ab(\g)$, respectively.

%A related interesting question is: How to characterise $\Ima(f_1)$ and $\Ima(f_2)$?
\end{rmk}

\end{document}